\documentclass[11pt]{article}
\usepackage{amsmath,amsthm,amsfonts,amssymb}
\usepackage{verbatim}
\usepackage{latexsym}
\usepackage{nicefrac}
\usepackage{color}
\usepackage{enumitem}
\usepackage{mdwlist}
\usepackage[british]{babel}
\usepackage[utf8]{inputenc}
\usepackage{soul}
\usepackage{gensymb}
\usepackage[maxbibnames=3,backend=bibtex,style=numeric,sorting=nty,firstinits=true,isbn=false,doi=false]{biblatex}
\usepackage{changepage}
\usepackage{mathtools}
\usepackage{hyperref}
\usepackage{authblk}
\hypersetup{pdfborder=0 0 0.6}

\newcommand{\R}{{\mathbb{R}}}
\newcommand{\C}{{\mathbb{C}}}

\newcommand{\N}{\mathbb{ N}}
\newcommand{\Z}{\mathbb{ Z}}
\newcommand{\T}{\mathbb{ T}}

\newcommand{\supp}{\operatorname{supp}}

\newtheorem{theorem}{Theorem}[section]
\newtheorem{lemma}[theorem]{Lemma}

\newtheorem{definition}[theorem]{Definition}
\newtheorem{corollary}[theorem]{Corollary}

\newtheorem{example}[theorem]{Example}
\newtheorem{examples}[theorem]{Examples}

\newtheorem{remarks}[theorem]{Remarks}

\begin{document}
\title{Uncertainty Principles as a Tool for STFT Phase Retrieval}
\author{David Bartusel}
\date{}
\affil{\small Fachgruppe Mathematik, RWTH Aachen University, D-52056 Aachen, Germany}
\affil{\small bartusel@mathematik.rwth-aachen.de}

\maketitle

\begin{abstract}
In the finite-dimensional setting, it is known that STFT phase retrieval is always possible when the window's ambiguity function does not vanish. However, it is not known how many zeros are allowed in the ambiguity function for the window still to allow phase retrieval.
In order to tackle this problem, we first consider a two-window approach where the second window equals the Fourier transform of the first window. This allows us to apply the uncertainty principle in order to obtain sufficient conditions for phase retrieval.
Using the relation between STFT phase retrieval and ambiguity sampling, we can prove sufficient conditions for the single-window phase retrieval problem, showing that only approximately eight ninths of the entries of the window's ambiguity function (and only three quarters in prime dimensions) are required to be nonzero.
\end{abstract}



\section{Introduction}

In the phase retrieval problem for the finite-dimensional short-time Fourier transform (STFT), mainly motivated by the field of ptychography \cite{FKMS}, one seeks to recover a signal $f\in\C^d$ from the absolute values of
\begin{equation*}
V_gf(k,l)=\sum_{j=0}^{d-1} f_j\overline{g_{j-k}}e^{\nicefrac{-2\pi i jl}{d}}\qquad (k,l\in\mathbb{Z}_d),
\end{equation*}
where $g\in\C^d$ is a known window function.

Aside from stability estimates for specific window and signal classes \cite{AWdis, CHLSS}, a main interest lies in the question which windows allow for the reconstruction of \textit{every} signal.
Conditions for suitable windows are usually given in terms of the support of the ambiguity function $V_gg$. It is well-known that phase retrieval is always possible if the window's ambiguity function vanishes nowhere \cite{BF}.
However, as pointed out in \cite{phasebook}, there are only few answers to the question how many (and which) entries of $V_gg$ are allowed to vanish for $g$ still to do phase retrieval: The authors of \cite{BF} have shown that phase retrieval is impossible whenever a full column $V_gg(k,\cdot)$ of the window's ambiguity function vanishes. Using the time-frequency duality of the problem (see Lemma \ref{lm:STFTid} \ref{STFTdual} below), it is easy to see that phase retrieval is equally impossible in the case of a vanishing row $V_gg(\cdot,l)$ (cf. also \cite{MR-backgr} for the case $l=0$). On the other hand, it is shown in \cite{PRinj} that there are certain subsets $M$ of cardinality $1$ or $2$ such that the condition $\supp\left(V_gg\right)=\Z_d^2\setminus M$ is sufficient for phase retrieval. The main goal of this paper is to improve this result by identifying sets $M=M_d$ of scaling cardinality $\left|M_d\right|\propto d^2$ such that $\supp\left(V_gg\right)=\Z_d^2\setminus M_d$ is still sufficient for phase retrieval.\\

As a relaxation of the \textit{global} phase retrieval problem, it is also common to study the question which signals can be recovered when using a window from a prescribed window class. Usually, the signals that can \textit{not} be recovered show some kind of localization \cite{AWdis,PRinj,CHLSS}.
However, by the uncertainty principle, a signal $f$ and its Fourier transform $\hat{f}$ cannot both be well-localized, which gives rise to the idea that combining the measurements $\left|V_gf\right|$ and $\left|V_g\hat{f}\right|$ might improve the chance of recovery significantly.
Another application of the time-frequency duality of the problem shows that this can be related to the recovery of $f$ from the measurements $\left|V_gf\right|$ and $\left|V_{\hat{g}}f\right|$. This leads to the question of multi-window STFT phase retrieval, which is another relaxation of the initial (single-window) phase retrieval problem, studied e.g. in \cite{CHLSS,WZ}. The more recent article \cite{AYFrF} addresses the multi-window approach in a continuous setting but is very similar to our approach in the sense that the authors study windows that are related to each other by rotation in the time-frequency plane (with the Fourier transform corresponding to a rotation of the time-frequency plane by 90\degree).

\subsection{Structure of the paper}

We will begin by discussing the two-window phase retrieval problem using a pair of windows $\left(g,\hat{g}\right)$. Using the uncertainty principle, we show that under the assumption $$\{-m,\dots,m\}\times\Z_d\subseteq\supp\left(V_gg\right)$$ for some $m>\frac{d-1}{3}$ (or $m>\frac{d-1}{4}$ if $d$ is prime), every $f\in\C^d$ can be recovered from $\left|V_gf\right|$ or from $\left|V_{\hat{g}}f\right|$ (Theorem \ref{thm:strtw-suff}). We will also give a straightforward example of a window satisfying $\supp\left(V_gg\right)=\{-m,\dots,m\}\times\Z_d$ (Example \ref{ex:exgtw}).
Additionally, we show that the uniqueness result does not carry over to the continuous setting of signals and windows in $L^2(\R)$ (Theorem \ref{thm:negcont})\\

Reinterpreting the two-window result in terms of sampling from $V_ff$, we show that every signal $f$ can be recovered from $V_ff\vert_{\Omega}$, where
\begin{equation*}
\Omega=\left(\{-m,\dots,m\}\times\Z_d\right)\cup\left(\Z_d\times\{-m,\dots,m\}\right)
\end{equation*}
for some $m>\frac{d-1}{3}$; if $d$ is prime it again suffices to take $m>\frac{d-1}{4}$ (Corollary \ref{cor:sampl-hole}). This also implies that a single window $g$ does phase retrieval whenever $\Omega\subseteq\supp\left(V_gg\right)$ (Corollary \ref{cor:PR-hole}).

For odd dimensions $d=2p+1$, we will give examples of windows satisfying\linebreak $\supp\left(V_gg\right)=\Z_d^2\setminus\{p,p+1\}$ or even $\supp\left(V_gg\right)=\Z_d^2\setminus\{p-1,p,p+1,p+2\}$ and thus doing phase retrieval by the previous result when $p$ is large enough (Examples \ref{ex:constr-windowhole}).

\subsection{Notation}

We fix a dimension $d\geq 2$. For a signal (or window) $f$, we write $f\in\C^d$ for short but actually think of the signal as a function $f:\Z_d\to\C^d$, where $\Z_d:=\Z/d\Z$.
For $j\in\Z$, we write $f_j:=f([j])$, i.e. all indices are interpreted modulo $d$.
For integers $k_1<k_2$, we will denote the range of indices between $k_1$ and $k_2$ as $\left\{k_1,\dots,k_2\right\}\coloneqq\{j\,|\,k_1\leq j\leq k_2\}\coloneqq\left\{[j]\,\middle|\,k_1\leq j\leq k_2\right\}$.

For $f\in\C^d$, we define the Fourier transform $\hat{f}:=\mathcal{F}f$ of $f$ by
\begin{equation*}
\hat{f}_l\coloneqq\sum_{j=0}^{d-1}f_jz_l^j
\end{equation*}
for all $l\in\Z_d$, where $z_l\coloneqq e^{\nicefrac{-2l\pi i}{d}}$.
Moreover, for $k,l\in\Z_d$, we define the translation $T_kf$ of $f$ by $\left(T_kf\right)_j=f_{j-k}$ for all $j\in\Z_d$ as well as the modulation $M_lf$ of $f$ by $\left(M_lf\right)_j=f_j\cdot z_l^{-j}$ for all $j\in\Z_d$.

For any finite set $M$, we use $|M|$ to denote the cardinality of $M$ and for a function $h:M\to\C$ (in the case where $M\neq\emptyset$), we denote the support of $h$ by $\supp(h)\coloneqq \{m\in M\,|\,h(m)\neq 0\}$.

Finally, we define the torus $\T\coloneqq\{z\in\mathbb{C}\,|\,|z|=1\}$.

\section{Strong two-window STFT phase retrieval}\label{sec:tw}

In this section, we show a sufficient condition for a window $g\in\C^d$ such that every signal $f\in\C^d$ can either be recovered from the measurement $\left|V_gf\right|$ or from the measurement $\left|V_{\hat{g}}f\right|$.
While this problem is interesting in its own right as a stronger form of multi-window phase retrieval, the result will be particularly useful when proving sufficient conditions for the single-window problem.\\

We begin by recalling the definition and some of the basic properties of the (finite-dimensional) short-time Fourier transform.

\begin{definition}
Let $f,g\in\C^d$. The short-time Fourier transform (STFT) of the signal $f$ with respect to the window $g$ is given by
\begin{equation*}
V_gf(k,l)=\left\langle f,M_lT_kg\right\rangle=\sum_{k=0}^{d-1} f_j\overline{g_{j-k}}e^{\nicefrac{-2\pi i jl}{d}}\qquad (k,l\in\mathbb{Z}_d).
\end{equation*}
\end{definition}

The following lemma contains some well-known but useful relations concerning the STFT.

\begin{lemma}\label{lm:STFTid}
Let $f,g\in\C^d$ as well as $k,l\in\Z_d$. Then, the following hold true.
\begin{enumerate}[label=\alph*)]
\item\label{ambigrel} $\mathcal{F}\left|V_gf\right|^2(l,-k)=d\cdot V_ff(k,l)\cdot V_gg(k,l)$
\item\label{ambigsym} $V_gg(-k,-l)=e^{\nicefrac{-2\pi ikl}{d}}\cdot\overline{V_gg(k,l)}$
\item\label{STFTdual} $V_{\hat{g}}{\hat{f}}(k,l)=d\cdot e^{\nicefrac{-2\pi ikl}{d}}\cdot V_gf(-l,k)$
\end{enumerate}
In \ref{ambigrel}, $\mathcal{F}$ denotes the two-variable Fourier transform.
\end{lemma}

\begin{proof}
We refer to \cite{GKR} for a proof of \ref{ambigrel}.
The symmetry identity \ref{ambigsym} and time-frequency duality \ref{STFTdual} are immediate consequences of the commutation relations between translation, modulation and Fourier transform as well as Plancherel's theorem. 
\end{proof}

Having introduced the STFT, we are now ready to state the problems at interest.
\pagebreak

\begin{definition}
Let $g,g_1,g_2\in\C^d$.
\begin{enumerate}[label=\alph*)]
\item A signal $f\in\C^d$ is called \textit{phase retrievable} w.r.t $g$ if
\begin{equation*}
\left|V_g\tilde{f}\right|=\left|V_gf\right|\quad\Rightarrow\quad \exists\gamma\in\T:~\tilde{f}=\gamma f
\end{equation*}
holds for all $\tilde{f}\in\C^d$.
\item The window $g$ is said to \textit{do phase retrieval} if every $f\in\C^d$ is phase retrievable w.r.t. $g$.
\item The windows $g_1,g_2\in\C^d$ are said to \textit{do strong two-window phase retrieval} if every $f\in\C^d$ is phase retrievable w.r.t. at least one of the windows $g_1$ or $g_2$.
\end{enumerate}
\end{definition}

The fact that we only ask a signal to be recovered up to a global phase factor is common and necessary whenever phase retrieval for linear operators (such as the STFT) is considered since the multiplication with such a constant clearly does not change the measurement.

We note that the ambiguity function relation Lemma \ref{lm:STFTid} \ref{ambigrel} is the reason for the well-known fact (cf. e.g. \cite[Lemma 2.1.3]{PRinj}) that the information obtained by measuring $\left|V_gf\right|$ is equivalent to sampling $V_ff$ on $\supp\left(V_gg\right)$, showing in particular that all phase retrieval properties of a window $g$ depend solely on the support of its ambiguity function.

In contrast to the classical (weak) two-window phase retrieval problem where the measurements $\left|V_{g_1}f\right|$ and $\left|V_{g_2}f\right|$ are combined, \textit{strong} two-window phase retrieval asks that every signal be recovered either from $\left|V_{g_1}f\right|$ or from $\left|V_{g_2}f\right|$. While weak versions of the results in this section would still suffice to prove the results in Section \ref{sec:ow}, we choose to state them in their strong versions in order to emphasize the effect of the uncertainty principle.\\

We will now focus on the strong-window phase retrieval problem with a pair of windows of the form $\left(g,\hat{g}\right)$ for some suitable $g\in\C^d$.
In order to motivate this choice, we recall that phase retrievability of a signal $f$ is related to ``holes'' in $\supp(f)$ when using a certain class of short windows (cf. \cite[Corollary 2.5]{AWdis}, \cite[Theorem 2.2.3]{PRinj}).

\begin{theorem}\label{thm:PRhole}
Let $m\in\N_0$ and $g\in\C^d$ be such that $\supp\left(V_gg\right)=\left\{-m,\dots,m\right\}\times\Z_d$.
Then, a signal $f\in\C^d$ is \emph{not} phase retrievable w.r.t. $g$ if and only if there exist $k\in\Z_d$ as well as $m<j\leq d-m-1$ satisfying
\begin{equation*}
\left\{0,j\right\}\subseteq\supp\left(T_kf\right)\subseteq\{0,\dots, j-m-1\}\cup\{j,\dots,d-m-1\}.
\end{equation*}
\end{theorem}

It is clear that for larger $m$, only well-localized signals can fail to be phase retrievable. (In particular, every signal is phase retrievable once $m\geq\frac{d}{2}$, i.e. $\supp\left(V_gg\right)=\Z_d^2$.)
By the uncertainty principle, it is therefore reasonable to believe that whenever $m$ is large enough, for every $f\in\C^d$, at least one of the pair $\left(f,\hat{f}\right)$ is phase retrievable.
By the time-frequency duality, this leads directly to strong two-window phase retrieval for $g$ and $\hat{g}$.

\begin{lemma}\label{lm:strtw-dual}
Let $g\in\C^d$. Then, the following are equivalent.
\begin{enumerate}[label=(\roman*)]
\item The pair $\left(g,\hat{g}\right)$ does strong two-window phase retrieval.
\item For every $f\in\C^d$, at least one of the pair $\left(f,\hat{f}\right)$ is phase retrievable w.r.t. $g$.
\end{enumerate}
\end{lemma}

\begin{proof}
Since the Fourier transform is one-to-one, statement (i) is equivalent to the fact that for every $f\in\C^d$, the signal $\hat{f}$ is phase retrievable w.r.t. at least one of the windows $g$ or $\hat{g}$.
By Lemma \ref{lm:STFTid} \ref{STFTdual}, $\hat{f}$ being phase retrievable w.r.t. $\hat{g}$ is equivalent to $f$ being phase retrievable w.r.t. $g$. This shows that the two statements are equivalent.
\end{proof}

Now, we are ready to apply the uncertainty principle to obtain a sufficient condition for strong-window phase retrieval using a pair $\left(g,\hat{g}\right)$.

\begin{theorem}\label{thm:strtw-suff}
Let $m^{\ast}:=\frac{d-1}{4}$ if $d$ is prime and $m^{\ast}:=\frac{d-1}{3}$ otherwise. Moreover, let $m>m^{\ast}$ and $g\in\C^d$ be such that $\{-m,\dots,m\}\times\Z_d\subseteq\supp\left(V_gg\right)$. Then, the pair $\left(g,\hat{g}\right)$ does strong two-window phase retrieval.
\end{theorem}

\begin{proof}
Suppose that the pair $\left(g,\hat{g}\right)$ does not do strong two-window phase retrieval. By Lemma \ref{lm:strtw-dual}, this implies the existence of $f\in\C^d$ such that neither $f$ nor $\hat{f}$ is phase retrievable w.r.t. $g$, i.e. both $f$ and $\hat{f}$ satisfy the conditions given in Theorem \ref{thm:PRhole}.
In particular, $\hat{f}$ has at least $2m$ zeros and we can assume w.l.o.g. that $\supp(f)\subseteq\{0,\dots,d-m-1\}$. The latter implies
\begin{equation}\label{eq:Fourier-cut}
\hat{f}_l=\sum_{j=0}^{d-m-1} f_jz_l^j
\end{equation}
for every $l\in\Z_d$ and therefore, $\hat{f}$ can have at most $d-m-1$ zeros. Altogether, we have $2m\leq d-m-1$, which contradicts $m>\frac{d-1}{3}$.\\
If $d$ is prime, we can use the stronger uncertainty principle
\begin{equation*}
\left|\supp(f)\right|+\left|\supp\left(\hat{f}\right)\right|\geq d+1
\end{equation*}
from \cite[Theorem 1.1]{Tao-UP}. Since both $f$ and $\hat{f}$ have at least $2m$ zeros, it follows $d+1\leq 2d-4m$, contradicting $m>\frac{d-1}{4}$.
\end{proof}

We make a couple of remarks regarding the application of the uncertainty principle.

\begin{remarks}\mbox{}
\begin{enumerate}[label=\arabic*)]
\item In the proof of Theorem \ref{thm:strtw-suff}, one could also try to apply the classical Donoho-Stark uncertainty principle \cite[Theorem 1]{DS-unc}, which states that $$\left|\supp(f)\right|\cdot\left|\supp\left(\hat{f}\right)\right|\geq d,$$ instead of taking into account the specific positions of the zeros. However, this would only yield a contradiction when assuming $m>\frac{d-\sqrt{d}}{2}$ which is a strictly stronger assumption than $m>\frac{d-1}{3}$ once $d>4$.
\item We conjecture that the condition $m>\frac{d-1}{3}$ for the general case might be improved further by a more in-depth analysis of \eqref{eq:Fourier-cut}. In the proof of Theorem \ref{thm:strtw-suff}, we have not used the fact that there are at least another $m$ consecutive zero entries within $f_0,\dots,f_{d-m-1}$, giving the polynomial in question even more structure. While a maximum number of $d-m-1$ zeros in $\T$ can still be achieved under this restriction (e.g. for $f_{d-m-1}=1$, $f_0=-1$, and all other coefficients being zero), we do currently not know how many of those can in general be of the form $z_l$ for some $l$ and thus correspond to a zero of $\hat{f}$ as well as being arranged in such a way that they form two pairs of $m$  \emph{consecutive} zeros. On the other hand, the condition $m>\frac{d-1}{4}$ for prime dimensions is sharp as is shown by the following corollary.
\end{enumerate}
\end{remarks}

\begin{corollary}\label{cor:primesharp}
Assume that $d$ is prime. Moreover, let $m\in\N$ and $g\in\C^d$ be such that $\supp\left(V_gg\right)=\{-m,\dots,m\}\times\Z_d$. Then, the pair $\left(g,\hat{g}\right)$ does strong two-window phase retrieval if and only if $m>\frac{d-1}{4}$.
\end{corollary}

\begin{proof}
By Theorem \ref{thm:strtw-suff}, it remains to show that strong two-window phase retrieval is impossible in the case $m\leq\frac{d-1}{4}$. Since this is clear for $d=2$, we can assume that $d$ is odd.
In this case, we consider the set
\begin{equation*}
S\coloneqq\left\{0,\dots,\left\lfloor\frac{d+1}{4}\right\rfloor-1\right\}\cup\left\{\left\lfloor\frac{d+1}{4}\right\rfloor+m,\dots,\frac{d+1}{2}+m-1\right\}.
\end{equation*}
Clearly, we have $|S|=\frac{d+1}{2}$. By \cite[Theorem 1.1]{Tao-UP}, there exists $f\in\C^d$ satisfying\linebreak $\supp(f)=\supp\left(\hat{f}\right)=S$. Since $\max(S)\leq d-1-m$, both $f$ and $\hat{f}$ satisfy the condition from Theorem \ref{thm:PRhole} and by Lemma \ref{lm:strtw-dual}, the pair $\left(g,\hat{g}\right)$ cannot do strong two-window phase retrieval.
\end{proof}

The role of the pair $\left(g,\hat{g}\right)$ in Theorem \ref{thm:strtw-suff} can be interpreted as follows: Depending on the signal, we can choose whether to reconstruct $f$ in time domain or in frequency domain. By the uncertainty principle, at least one of these is always possible, given the assumptions on $g$.
Since $\supp(f)$ alone determines whether $f$ can be reconstructed in time domain, one can use a ``test measurement'' of $\left|V_gf(0,\cdot)\right|$ in order to decide in which domain to perform the reconstruction.\\

For Theorem \ref{thm:strtw-suff} to carry any meaning, it is necessary to show the existence of a window $g\in\C^d$ satisfying $\{-m,\dots,m\}\times\Z_d\subseteq\supp\left(V_gg\right)$ for some $m^{\ast}<m<\frac{d-1}{2}$, where $m^{\ast}$ is as in the theorem. (For $m\geq\frac{d-1}{2}$, the condition already implies $\supp\left(V_gg\right)=\Z_d^2$, in which case $g$ is known to do phase retrieval.) Such an $m$ exists if and only if $d\geq 6$.

In order to construct suitable windows, we turn to the idea of exponential windows as in \cite{PhD-DB,IVW,IMPV}.

\begin{example}\label{ex:exgtw}
Let $d\geq 6$ as well as $m^{\ast}<m<\frac{d-1}{2}$, where $m^{\ast}$ is as in Theorem \ref{thm:strtw-suff}. Let $V\subseteq\{0,\dots,m\}$ be such that $$\left\{|j-k|\,\middle|\,j,k\in V\right\}=\{0,\dots,m\}$$ (e.g. $V=\{0,\dots,m\}$). Then, the window $g\in\C^d$, defined by
\begin{equation*}
g_j:=\begin{cases}2^j,&\text{if } j\in V,\\0,&\text{otherwise},\end{cases}
\end{equation*}
satisfies $\supp\left(V_gg\right)=\{-m,\dots,m\}\times\Z_d$.
\end{example}

\begin{proof}
Clearly, we have $V_gg(k,\cdot)\equiv 0$ whenever $k\notin\{-m,\dots,m\}$. By Lemma \ref{lm:STFTid} \ref{ambigsym}, it remains to show that $V_gg(k,l)\neq 0$ holds for all $k\in\{0,\dots,m\}$ and all $l\in\{0,\dots,d-1\}$.
For fixed $k\in\{0,\dots,m\}$, we define $$S_k:=\left\{k\leq j\leq m\,\middle|\,g_j\overline{g_{j-k}}\not\equiv 0\right\}\neq\emptyset$$ as well as $j_k:=\max\left(S_k\right)$. For every $l\in\Z_d$, it follows
\begin{align*}
\left|V_gg(k,l)\right|	&=\left|\sum_{j\in S_k} g_j\overline{g_{j-k}}e^{\nicefrac{-2\pi ijl}{d}}\right|\\
					&\geq\left|g_{j_k}\right|\cdot\left|g_{j_k-k}\right|-\sum_{j\in S_k\setminus\left\{j_k\right\}}\left|g_j\right|\cdot\left|g_{j-k}\right|\\
					&=\frac{1}{2^k}\cdot\left(4^{j_k}-\sum_{j\in S_k\setminus\left\{j_k\right\}}4^j\right)\geq\frac{1}{2^k}\cdot\left(4^{j_k}-\sum_{j=0}^{j_k-1} 4^j\right)\\
					&=\frac{1}{2^k}\cdot\left(4^{j_k}-\frac{4^{j_k}-1}{3}\right)>0.
\end{align*}
\end{proof}

The example shows that strong two-window STFT phase retrieval is possible for windows $\left(g,\hat{g}\right)$ with $g\in\R^d$ while it is known that single-window STFT phase retrieval is impossible for real windows in even dimensions \cite[Theorem 7.1.13]{PhD-DB}.

We also note that by using windows of exponential type, we do not require our support set $V$ to permit the existence of a number $r\in\mathbb{N}$ such that every $j\in\{1,\dots,m\}$ appears exactly $r$ times as a difference of elements in $V$, which would be similar to the construction in \cite[Proposition 2.2]{BF}. This allows for a higher degree of freedom in the choice of $V$.\\

We end the section by noting that Theorem \ref{thm:strtw-suff} does not translate to the continuous setting. The proof is very similar to that of Corollary \ref{cor:primesharp}. We use the standard notation and refer readers who are not familiar with the continuous short-time Fourier transform and its properties to \cite{Groech}.

\begin{theorem}\label{thm:negcont}
Let $R>0$ and $g\in L^2(\mathbb{R})$ be such that $g\equiv 0$ a.e. outside some interval $[a,a+R]$ with $a\in\R$, or equivalently $V_gg\equiv 0$ outside $[-R,R]\times\R$.
Then, there exist $f,f_1,f_2\in L^2(\mathbb{R})$ satisfying $f_1,f_2\notin\T\cdot f$ as well as
\begin{equation*}
\left|V_g f_1\right|=\left|V_gf\right|\quad\text{and}\quad\left|V_{\hat{g}}f_2\right|=\left|V_{\hat{g}}f\right|,
\end{equation*}
i.e. $f$ is neither phase retrievable w.r.t. $g$ nor phase retrievable w.r.t. $\hat{g}$.
\end{theorem}

\begin{proof}
By \cite{AB-tfgaps}, there exists $0\not\equiv h\in L^2(\R)$ satisfying $h,\hat{h}\equiv 0$ a.e. on $[0,R]$. (In fact, the subspace of functions with this property is infinite-dimensional.) It is well-known (cf. e.g. \cite{PRinj}) that neither $h$ nor $\hat{h}$ are phase retrievable w.r.t. $g$ under the given assumptions.
Letting $f:=\hat{h}$ and using the fact that $\left|V_{\hat{g}}f(x,\omega)\right|=\left|V_gh(-\omega,x)\right|$, it follows that $f$ is neither phase retrievable w.r.t. $g$ nor phase retrievable w.r.t. $\hat{g}$.
\end{proof}

Theorem \ref{thm:negcont} only states that \textit{strong} two-window phase retrieval is impossible for the considered window class. We do not know whether there exist windows $g\in L^2(\R)$ whose ambiguity function vanishes outside a vertical strip $[-R,R]\times\R$ such that every signal $f\in L^2(\R)$ can be recovered (up to global phase) from the \textit{combined} measurements $\left|V_gf\right|$ and $\left|V_{\hat{g}}f\right|$. The same applies to Corollary \ref{cor:primesharp}.

\section{Single-window STFT phase retrieval}\label{sec:ow}

We will now return our attention to the single-window phase retrieval problem.
In order to do so, we recall that the STFT phase retrieval problem is essentially a sampling problem in the time-frequency plane, in the sense that we need to recover $f$ from samples of its quadratic time-frequency representation $V_ff$.
Since the same applies to the two-window problem, Theorem \ref{thm:strtw-suff} leads to the following sampling result. As mentioned before, a weak version of the theorem would suffice for the proof.

\begin{corollary}\label{cor:sampl-hole}
Let $m>m^{\ast}$, where $m^{\ast}$ is as in Theorem \ref{thm:strtw-suff}, and define
\begin{equation*}
\Omega\coloneqq\left(\{-m,\dots,m\}\times\Z_d\right)\cup\left(\Z_d\times\{-m,\dots,m\}\right).
\end{equation*}
If $f,\tilde{f}\in\C^d$ satisfy $V_{\tilde{f}}\tilde{f}\vert_{\Omega}=V_ff\vert_{\Omega}$, there exists $\gamma\in\T$ satisfying $\tilde{f}=\gamma f$.
\end{corollary}

\begin{proof}
We need to show that every $f$ can be reconstructed up to global phase from $V_ff\vert_{\Omega}$.\\
For $d<6$, this is clear since we have $\Omega=\Z_d^2$ in this case.\\
For $d\geq 6$, we choose $g\in\C^d$ satisfying $\supp\left(V_gg\right)=\{-m,\dots,m\}\times\Z_d$. The existence of $g$ is shown in Example \ref{ex:exgtw} and the pair $\left(g,\hat{g}\right)$ does (strong) two-window phase retrieval by Theorem \ref{thm:strtw-suff}.
By Lemma \ref{lm:STFTid} \ref{STFTdual}, we have
\begin{equation*}
\supp\left(V_{\hat{g}}\hat{g}\right)=\Z^d\times\{-m,\dots,m\}
\end{equation*}
and consequently, both $\supp\left(V_gg\right)$ and $\supp\left(V_{\hat{g}}\hat{g}\right)$ are subsets of $\Omega$.
Given $V_ff\vert_{\Omega}$, we can therefore reconstruct $V_ff\cdot V_gg$ as well as $V_ff\cdot V_{\hat{g}}\hat{g}$ and thus both $\left|V_gf\right|$ and $\left|V_{\hat{g}}f\right|$.
Since the pair $\left(g,\hat{g}\right)$ does (strong) two-window phase retrieval, this allows us to reconstruct $f$ up to global phase.
\end{proof}

The usual relation between STFT phase retrieval and sampling in the time-frequency plane immediately yields the following reformulation.

\begin{corollary}\label{cor:PR-hole}
Let $m>m^{\ast}$, where $m^{\ast}$ is as in Theorem \ref{thm:strtw-suff}, and $g\in\C^d$ be such that
\begin{equation*}
\left(\{-m,\dots,m\}\times\Z_d\right)\cup\left(\Z_d\times\{-m,\dots,m\}\right)\subseteq\supp\left(V_gg\right).
\end{equation*}
Then, $g$ does STFT phase retrieval.
\end{corollary}

For $d\geq 6$, we can choose $m^{\ast}<m<\frac{d-1}{2}$, in which case $$\left(\{-m,\dots,m\}\times\Z_d\right)\cup\left(\Z_d\times\{-m,\dots,m\}\right)$$ is a proper subset of $\Z_d^2$. We can then rewrite
\begin{equation*}
\left(\{-m,\dots,m\}\times\Z_d\right)\cup\left(\Z_d\times\{-m,\dots,m\}\right)=\Z_d^2\setminus\{m+1,\dots,d-(m+1)\}^2.
\end{equation*}
Since we can always choose $m$ to be smaller than or equaler to $\frac{d+2}{3}$ (smaller than or equal to $\frac{d+3}{4}$ when $d$ is prime), Corollary \ref{cor:PR-hole} shows that a window is allowed to have its ambiguity function vanish on a set $M_d$ of cardinality
\begin{equation*}
\left|M_d\right|\geq\left(d-2\cdot\frac{d+2}{3}-1\right)^2=\left(\frac{d-7}{3}\right)^2\simeq \frac{d^2}{9},
\end{equation*}
i.e. a ninth of the time-frequency plane, in the general case and of cardinality
\begin{equation*}
\left|M_d\right|\geq\left(d-2\cdot\frac{d+3}{4}-1\right)^2=\left(\frac{d-5}{2}\right)^2\simeq \frac{d^2}{4},
\end{equation*}
i.e. a quarter of the time-frequency plane, when $d$ is prime, without losing the ability to do phase retrieval. This clearly generalizes the existing result that $V_gg$ is allowed to vanish at $\left(\frac{d}{2},\frac{d}{2}\right)$ in the case where $d\geq 4$ is even \cite[Example 7.1.7]{PhD-DB}.
In light of the cyclic structure, we recognize that the ``hole'' in the time-frequency plane is located in the area of (in terms of modulus) large time and frequency parameters. Therefore, Corollary \ref{cor:sampl-hole} means that every signal can be reconstructed from its ambiguity function by using only the information where either the time or the frequency index are small.\\
If one is purely interested in sampling from $V_ff$ (without the constraint to represent the sampling set as the support of a window's ambiguity function), one can ignore all samples $V_ff(k,l)$ for $\frac{d}{2}<k\leq d-1$ in Corollary \ref{cor:sampl-hole} due to the symmetry of $V_ff$ and the symmetry of $\Omega$, showing that we can choose a sampling set of cardinality $\simeq\frac{4}{9}d^2$ in the general case and of cardinality $\simeq\frac{3}{8}d^2$ when $d$ is prime.

It is important to note that the sharpness of the bound on $m$ in the prime case for the strong two-window problem (Corollary \ref{cor:primesharp}) does not necessarily carry over to the results of this section.\\

While Corollary \ref{cor:sampl-hole} is interesting in its own right, Corollary \ref{cor:PR-hole} relies on the existence of a suitable window in order to extend the previously known results for phase retrieval.
The remainder of this section will be devoted to the construction of such windows.

For odd dimensions $d=2p+1$, we suggest the ansatz of a window $g\in\C^d$ satisfying\linebreak $\supp(g)\subseteq\{0,\dots,p+n\}$ for some $n<p$ in order to simplify calculations. In this case, we have
\begin{equation*}
V_gg(p,l)=\sum_{j=0}^{n-1}g_j\overline{g_{j+p+1}}z_l^j+\sum_{j=p}^{p+n}g_j\overline{g_{j-p}}z_l^j
\end{equation*}
for all $l\in\Z_d$. Our first goal is to ensure that $V_gg(p,l)$ vanishes for $l\in\{p+1-r,\dots,p+r\}$ for some $r\in\N$. This can only be achieved when the polynomial
\begin{equation*}
q_1(z):=\prod_{j=1}^r\left(z-z_{p+j}\right)\left(z-z_{p+1-j}\right)
\end{equation*}
is a divisor of
\begin{equation*}
q(z):=\sum_{j=0}^{n-1}g_j\overline{g_{j+p+1}}z^j+\sum_{j=p}^{p+n}g_j\overline{g_{j-p}}z^j.
\end{equation*}
Consequently, we are looking for a polynomial $q_2$ of degree $p+n-2r$ such that the coefficients $a_n,\dots,a_{p-1}$ in
\begin{equation*}
\sum_{j=0}^{p+n} a_jz^j:=q_1(z)q_2(z)
\end{equation*}
vanish. Assuming w.l.o.g. that $q_2$ is normalized, this leads to a system of linear equations with $p+n-2r$ variables, consisting of $p-n$ equations. Aiming for a quadratic system, we suggest to take $n=r$.

We can now hope for the system to be regular. Solving it gives us the coefficients $a_j$ above and we can choose the coefficients $g_0,\dots,g_r$ as well as $g_p,\dots,g_{p+r}$ accordingly. The remaining $p-r-1$ entries of $g$ then have to be chosen in such a way that $V_gg(k,l)=0$ holds exactly for $k\in\{p+1-r,\dots,p-1\}$ and $l\in\{p+1-r,\dots,p+r\}$.
Working from $k=p-1$ to $k=p+1-r$, two more entries of $g$ will appear in $V_gg(k,\cdot)$ with each step. In order to leave at least one coefficient to ensure that $V_gg(k,\cdot)$ has no zeros for $0\leq k\leq p-r$ therefore requires $2(r-1)<p-r-1$, i.e. $r<\frac{p+1}{3}$.
In order to construct a window that exploits Corollary \ref{cor:PR-hole} to its full extent, we would like to choose $r=\frac{d-(2m+1)}{2}$, where $m$ is the smallest integer greater than $m^{\ast}$ and $m^{\ast}$ is as in Theorem \ref{thm:strtw-suff}.
This amounts to $r\approx\frac{p}{3}$ in the general case and to $r\approx\frac{p}{2}$ in the prime case, suggesting that our ansatz might be suitable to construct windows that exhaust the possibilities of Corollary \ref{cor:PR-hole} in the general case but not in the prime case.\\

Of course, the approach outlined above is still very vague and we do not know whether it can be carried out uniformly. Moreover, the approach does not work for even dimensions: Starting with $V_gg\left(\frac{d}{2},\cdot\right)$ and prescribing its zeros in the same way as above inevitably leads to a real-valued system of linear equations, thus leading to real-valued coefficients in $V_gg\left(\frac{d}{2},\cdot\right)$. It is already known that phase retrieval is not possible in this case, cf. \cite[Theorem 7.1.13]{PhD-DB}.\\

However, for odd dimensions and specific choices of $r$, we are able to apply the approach in order to construct suitable windows.

\begin{examples}\label{ex:constr-windowhole}
Let $p\in\N$ as well as $d=2p+1$.
\begin{enumerate}[label=\alph*)]
\item\label{ex:hole-2x2} If $p\geq 3$, the window $g\in\R^d\subseteq\C^d$, defined by
\begin{equation*}
g_j:=\begin{cases} -\frac{1}{2\cos\left(\frac{2p^2\pi}{d}\right)},&\text{if } j=1,\\
				c,&\text{if } j=p-1,\\
				-2\cos\left(\frac{2p^2\pi}{d}\right),&\text{if } j=p+1,\\
				0,&\text{if } j\in\{p+2,\dots,d-1\},\\
				1,&\text{otherwise},\end{cases}
\end{equation*}
for all $j\in\Z_d$, satisfies $\supp\left(V_gg\right)=\Z_d^2\setminus\{p,p+1\}^2$ for all but finitely many choices of $c\in\mathbb{R}$.
\item\label{ex:hole-4x4} If $p\geq 6$, the window $g\in\R^d\subseteq\C^d$, defined by
\begin{equation*}
g_j:=\begin{cases}-\frac{4\cos\left(\frac{\pi}{d}\right)\cos\left(\frac{p\pi}{d}\right)}{2\cos\left(\frac{\pi}{d}\right)-1},&\text{if } j=0,\\
				2\cos\left(\frac{\pi}{d}\right)\cdot(-1)^{p+1},&\text{if } j=1,\\
				-\frac{1}{2\cos\left(\frac{p\pi}{d}\right)},&\text{if } j=2,\\
				y_p,&\text{if }j=3,\\
				0,&\text{if } j\in\{p-6,\dots,p-3\}\setminus\{0,\dots,3\},\\
				c,&\text{if } j=p-2,\\
				\frac{2\cos\left(\frac{\pi}{d}\right)-1}{2\cos\left(\frac{\pi}{d}\right)}\cdot y_p,&\text{if } j=p-1,\\
				-\frac{2\cos\left(\frac{\pi}{d}\right)-1}{4\cos\left(\frac{\pi}{d}\right)\cos\left(\frac{p\pi}{d}\right)},&\text{if } j=p,\\
				\left(2\cos\left(\frac{\pi}{d}\right)-1\right)\cdot(-1)^{p+1},&\text{if } j=p+1,\\
				-2\cos\left(\frac{p\pi}{d}\right),&\text{if } j=p+2,\\
				0,&\text{if } j\in\{p+3,\dots,d-1\},\\
				1,&\text{otherwise},\end{cases}
\end{equation*}
for all $j\in\Z_d$, where
\begin{equation*}
y_p:=\left(2\cos\left(\frac{\pi}{d}\right)-1\right)\cdot(-1)^{p+1}\cdot\frac{\cos\left(\frac{p\pi}{d}\right)+\cos\left(\frac{(p-1)\pi}{d}\right)}{\cos\left(\frac{(p-1)\pi}{d}\right)+\cos\left(\frac{(p-4)\pi}{d}\right)}\cdot\frac{1}{4\cos^2\left(\frac{p\pi}{d}\right)},
\end{equation*}
satisfies $\supp\left(V_gg\right)=\Z_d^2\setminus\{p-1,p,p+1,p+2\}^2$ for all but finitely many choices of $c\in\mathbb{R}$.
\end{enumerate}
\end{examples}

\begin{proof}\mbox{}
\begin{enumerate}[label=\alph*)]
\item The fact that $d$ is odd implies $\cos\left(\frac{2p^2\pi}{d}\right)\neq 0$ and thus, $g$ is well-defined. By Lemma \ref{lm:STFTid} \ref{ambigsym}, it suffices to compute $V_gg(k,l)$ for $0\leq k\leq p$.

For $k=p$ and every $l\in\Z_d$, we compute
\begin{align*}
V_gg(p,l)&=g_0g_{p+1}+g_pg_0z_l^p+g_{p+1}g_1z_l^{p+1}=z_l^{p+1}+z_l^p-2\cos\left(\frac{2p^2\pi}{d}\right)\\
		&=2\operatorname{Re}\left(z_l^p\right)-2\cos\left(\frac{2p^2\pi}{d}\right)=2\cos\left(\frac{2lp\pi}{d}\right)-2\cos\left(\frac{2p^2\pi}{d}\right).
\end{align*}
This implies that $V_gg(p,l)=0$ holds if and only if $lp\equiv p^2~\operatorname{mod} d$ or $lp\equiv -p^2~\operatorname{mod} d$ and since $\operatorname{gcd}(p,d)=1$, this is equivalent to $l\in\{\pm p\}=\{p,p+1\}$.

For $0\leq k\leq p-1$, we note that
\begin{equation*}
V_gg(k,l)=\sum_{j=k}^{p+1}g_jg_{j-k}z_l^j
\end{equation*}
for all $l\in\Z_d$. Consequently, we have $V_gg(k,l)=a_{kl}c^{1+\delta_{k,0}}+b_{kl}$, where $a_{kl},b_{kl}\in\C$ and $\delta_{k,0}$ is the Kronecker delta.
For each pair $(k,l)$, this can only be zero for finitely many $c\in\R$, as long as we can show that $a_{kl}\neq 0$.
This is clear for $k=0$, where $a_{kl}=z_l^{p-1}$ for all $l\in\Z_d$, as well as for $3\leq k\leq p-1$, where $a_{kl}=g_{p-1-k}z_l^{p-1}$ for all $l\in\Z_d$.
For $k\in\{1,2\}$, we have $a_{kl}=z_l^{p-1}\cdot\left(g_{p-1-k}+g_{p-1+k}z_l^k\right)$ and it can easily be verified that $g_{p-1-k}\neq -g_{p-1+k}$, implying $a_{kl}\neq 0$ since $d$ is odd.


\item The fact that $d\geq 6$ (as well as $\frac{p}{d}<\frac{1}{2}$) implies that all denominators are nonzero and therefore, $g$ is well-defined. Again, we only need to consider $V_gg(k,l)$ for $0\leq k\leq p$.

For $k=p$ and every $l\in\Z_d$, we compute
\begin{align*}
V_gg(p,l)&=g_0g_{p+1}+g_1g_{p+2}z_l+g_pg_0z_l^p+g_{p+1}g_1z_l^{p+1}+g_{p+2}g_2z_l^{p+2}\\
		&=(-1)^p\cdot 4\cos\left(\frac{\pi}{d}\right)\cos\left(\frac{p\pi}{d}\right)\left(z_l+1\right)\\
		&\qquad+z_l^p+2\cos\left(\frac{\pi}{d}\right)\left(2\cos\left(\frac{\pi}{d}\right)-1\right)z_l^{p+1}+z_l^{p+2}.
\end{align*}
Multiplying the right-hand side by $z_l^p$ shows that $V_gg(p,l)=0$ is equivalent to
\begin{align*}
2\cos\left(\frac{\pi}{d}\right)\left(2\cos\left(\frac{\pi}{d}\right)-1\right)+z_l+\overline{z_l}&=(-1)^{p+1}\cdot 4\cos\left(\frac{\pi}{d}\right)\cos\left(\frac{p\pi}{d}\right)\left(z_l^p+\overline{z_l^p}\right)\\
&\Leftrightarrow\\
\cos\left(\frac{\pi}{d}\right)\left(2\cos\left(\frac{\pi}{d}\right)-1\right)+\cos\left(\frac{2l\pi}{d}\right)&=(-1)^{p+1}\cdot 4\cos\left(\frac{\pi}{d}\right)\cos\left(\frac{p\pi}{d}\right)\cos\left(\frac{2pl\pi}{d}\right).
\end{align*}
In particular, we have $V_gg(p,-l)=0$ if and only if $V_gg(p,l)=0$.
Furthermore, since $\cos\left(\frac{2l\pi}{d}\right)=2\cos^2\left(\frac{\pi l}{d}\right)-1$ and $\cos\left(\frac{2pl\pi}{d}\right)=(-1)^l\cos\left(\frac{l\pi}{d}\right)$, there exist polynomials $q_1,q_2$ of degree $2$ such that $V_gg(p,l)=0$ is equivalent to $q_1\left(\cos\left(\frac{l\pi}{d}\right)\right)=0$ for all even $l$ and to $q_2\left(\cos\left(\frac{l\pi}{d}\right)\right)=0$ for all odd $l$.
As the elements $\cos\left(\frac{\pi l}{d}\right)$ are distinct for $l\in\{0,\dots,d-1\}$, this implies that $V_gg(p,l)=0$ can hold for at most four $l\in\{0,\dots,d-1\}$.
Using the properties of the cosine function, it can easily be verified that $V_gg(p,p-1)=V_gg(p,p)=0$ and by the symmetry established above, it follows $\supp\left(V_gg(p,\cdot)\right)=\Z_d\setminus\{p-1,p,p+1,p+2\}$ as desired.

For $k=p-1$ and every $l\in\Z_d$, we compute
\begin{align*}
V_gg(p-1,l)&=g_0g_{p+2}+g_{p-1}g_0z_l^{p-1}+g_pg_1z_l^p+g_{p+1}g_2z_l^{p+1}+g_{p+2}g_3z_l^{p+2}\\
		&=\frac{8\cos\left(\frac{\pi}{d}\right)\cos\left(\frac{p\pi}{d}\right)^2}{2\cos\left(\frac{\pi}{d}\right)-1}+(-1)^p\frac{2\cos\left(\frac{\pi}{d}\right)-1}{2\cos\left(\frac{p\pi}{d}\right)}\left(z_l^p+\overline{z_l^p}\right)\\
		&\quad~+(-1)^p\frac{2\cos\left(\frac{\pi}{d}\right)-1}{2\cos\left(\frac{p\pi}{d}\right)}\cdot\frac{\cos\left(\frac{p\pi}{d}\right)+\cos\left(\frac{(p-1)\pi}{d}\right)}{\cos\left(\frac{(p-1)\pi}{d}\right)+\cos\left(\frac{(p-4)\pi}{d}\right)}\left(z_l^{p-1}+\overline{z_l^{p-1}}\right)\\
		&=\frac{8\cos\left(\frac{\pi}{d}\right)\cos\left(\frac{p\pi}{d}\right)^2}{2\cos\left(\frac{\pi}{d}\right)-1}+(-1)^p\frac{2\cos\left(\frac{\pi}{d}\right)-1}{\cos\left(\frac{p\pi}{d}\right)}\cos\left(\frac{2pl\pi}{d}\right)\\
		&\quad~+(-1)^p\frac{2\cos\left(\frac{\pi}{d}\right)-1}{\cos\left(\frac{p\pi}{d}\right)}\cdot\frac{\cos\left(\frac{p\pi}{d}\right)+\cos\left(\frac{(p-1)\pi}{d}\right)}{\cos\left(\frac{(p-1)\pi}{d}\right)+\cos\left(\frac{(p-4)\pi}{d}\right)}\cos\left(\frac{2(p-1)l\pi}{d}\right),
\end{align*}
which implies that $V_gg(p-1,-l)=0$ holds if and only if $V_gg(p-1,l)=0$. Multiplying the right-hand side by
\begin{equation*}
(-1)^{p+l}\cdot\left(2\cos\left(\frac{\pi}{d}\right)-1\right)\cos\left(\frac{p\pi}{d}\right)\left(\cos\left(\frac{(p-1)\pi}{d}\right)+\cos\left(\frac{(p-4)\pi}{d}\right)\right)
\end{equation*}
shows that $V_gg(p-1,l)=0$ is equivalent to
\begin{align}
\begin{split}\label{eq:zeros-p-1}
&(-1)^{p+l}\cdot 8\cos\left(\frac{\pi}{d}\right)\cos\left(\frac{p\pi}{d}\right)^3\left(\cos\left(\frac{(p-1)\pi}{d}\right)+\cos\left(\frac{(p-4)\pi}{d}\right)\right)\\
={}&-\left(2\cos\left(\frac{\pi}{d}\right)-1\right)^2\cdot\Bigg(\left(\cos\left(\frac{p\pi}{d}\right)+\cos\left(\frac{(p-1)\pi}{d}\right)\right)\cos\left(\frac{3l\pi}{d}\right)\\
&\hphantom{\left(2\cos\left(\frac{\pi}{d}\right)-1\right)^2\cdot}\qquad +\left(\cos\left(\frac{(p-1)\pi}{d}\right)+\cos\left(\frac{(p-4)\pi}{d}\right)\right)\cos\left(\frac{l\pi}{d}\right)\Bigg).
\end{split}
\end{align}
Using the properties of the cosine function as well as $3p=d+p-1$ and\linebreak $3(p-1)=d+p-4$, this can easily be verified for $l\in\{p-1,p\}$, implying $V_gg(p-1,l)=0$ for all $l\in\{p-1,p,p+1,p+2\}$ by the symmetry established above.\\
It remains to show that $p-1$, $p$, $p+1$ and $p+2$ are the only zeros of $V_gg(p-1,\cdot)$. In order to do so, we recall the identity $\cos(3x)=4\cos^3(x)-3\cos(x)$.
Together with \eqref{eq:zeros-p-1}, this implies that there exists a polynomial $q$ of degree $3$ such that $V_gg(p-1,l)$ is equivalent to $q\left(\cos\left(\frac{\pi l}{d}\right)\right)=(-1)^{p+l}$.
In particular, we have
\begin{equation*}
q\left(\cos\left(\frac{(p-1)\pi}{d}\right)\right)=q\left(\cos\left(\frac{(p+1)\pi}{d}\right)\right)=-1
\end{equation*}
as well as
\begin{equation*}
q\left(\cos\left(\frac{p\pi}{d}\right)\right)=q\left(\cos\left(\frac{(p+2)\pi}{d}\right)\right)=1.
\end{equation*}
Since $q$ is continuous and $l\mapsto\cos\left(\frac{l\pi}{d}\right)$ is strictly decreasing on $[0,d-1]$, $q$ has a zero in each of the intervals
\begin{gather*}
\left(\cos\left(\frac{(p+2)\pi}{d}\right),\cos\left(\frac{(p+1)\pi}{d}\right)\right), \left(\cos\left(\frac{(p+1)\pi}{d}\right),\cos\left(\frac{p\pi}{d}\right)\right)\\\text{and}~\left(\cos\left(\frac{p\pi}{d}\right),\cos\left(\frac{(p-1)\pi}{d}\right)\right).
\end{gather*}
Since $q$ is a polynomial of degree $3$, it follows that $q$ is strictly decreasing on both $\left(-\infty,\cos\left(\frac{(p+2)\pi}{d}\right)\right)$ and $\left(\cos\left(\frac{(p-1)\pi}{d}\right),\infty\right)$, which implies $V_gg(p-1,l)\neq 0$ for all\linebreak $l\in\Z_d\setminus\{p-1,p,p+1,p+2\}$.

For $0\leq k\leq p-2$, we note that
\begin{equation*}
V_gg(k,l)=\sum_{j=k}^{p+2} g_jg_{j-k}z_l^j
\end{equation*}
for all $l\in\Z_d$. Consequently, we have $V_gg(k,l)=a_{kl}c^{1+\delta_{k,0}}+b_{kl}$, where $a_{kl},b_{kl}\in\C$ and $\delta_{k,0}$ is the Kronecker delta.

For each pair $(k,l)$, this can only be zero for finitely many $c\in\R$, as long as we can show that $a_{kl}\neq 0$.
This is clear for $k=0$, where $a_{kl}=z_l^{p-2}$ for all $l\in\Z_d$, as well as for $5\leq k\leq p-2$, where $a_{kl}=g_{p-2-k}z_l^{p-2}$ for all $l\in\Z_d$.
For $k\in\{1,2,3,4\}$, we have $a_{kl}=z_l^{p-2}\cdot\left(g_{p-2-k}+g_{p-2+k}z_l^k\right)$ and it can easily be verified that $g_{p-2-k}\neq -g_{p-2+k}$, implying $a_{kl}\neq 0$ since $d$ is odd.
\end{enumerate}
\vspace{-\baselineskip}
\end{proof}

For odd dimensions $d\geq 13$, Example \ref{ex:constr-windowhole} \ref{ex:hole-4x4}, combined with Corollary \ref{cor:PR-hole}, improves the possible number of zeros in the ambiguity function of windows doing phase retrieval from $2$ (as previously known) to $16$.
In light of the ansatz presented above, it appears likely that the number of possible zeros can be increased further in order to scale with $d^2$, at least for odd dimensions.
In fact, for dimensions $d\in\{7,9,13,15\}$, the examples already have the maximum number of zeros allowed in Corollary \ref{cor:PR-hole}.\\
Since windows with a single zero of the ambiguity function at $\left(\frac{d}{2},\frac{d}{2}\right)$ are already known for even dimensions, it is also reasonable to expect that this ``hole'' can be enlarged to have a scaling size for even dimensions as well.

\section{Conclusion and outlook}

With Theorem \ref{thm:strtw-suff}, we have shown that (strong) two-window phase retrieval is possible for a window concentrated in time ($g$) together with a window concentrated in frequency ($\hat{g}$), given that the respective supports are still large enough.
This is due to the uncertainty principle ensuring that a signal $f$ cannot have a large time gap and a large frequency gap at the same time.\\

When translating the result into the setting of reconstructing a signal from samples of its ambiguity function, this leads to the fact that recovery can be achieved by using a (symmetric) sampling set of cardinality $\simeq \frac{8}{9}d^2$ in the general case and of cardinality $\simeq \frac{3}{4}d^2$ in the prime case (Corollary \ref{cor:sampl-hole}).
Consequently, single-window phase retrieval is also possible when the support of the window's ambiguity function is a suitable set of the respective cardinality (Corollary \ref{cor:PR-hole}).
As mentioned before, it is still an open problem to construct windows with a suitable support for higher dimensions. Moreover, it would be interesting to provide a (sharp) lower bound for the necessary number of samples, i.e. the support size of the window's ambiguity function.\\

Another important aspect in phase retrieval is stability. While it is clear that a smaller ambiguity support of the window will negatively impact the stability of the phase retrieval problem, it would be interesting to quantify this phenomenon.
On the other hand, using the two-window approach with windows $g$ and $\hat{g}$ (both in the weak and in the strong sense) is likely to increase the stability of the problem, again due to the uncertainty principle. As with the negative results, further research can be directed at a quantification of this effect.
We note again that increasing stability by rotating the window within the time-frequency plane has already been considered in the continuous case \cite{AYFrF}.\\

We also suggest to apply the ideas of this article to the continuous setting. The continuous phase retrieval problem is very similar to the disrcete setting in many structural aspects. In particular, the uncertainty principle is also present in the continuous case.
While it was easy to show that Theorem \ref{thm:strtw-suff} does not carry over to the continuous setting, we do not know whether Corollaries \ref{cor:sampl-hole} and \ref{cor:PR-hole} have analogues for $L^2(\R)$. In this setting, the sampling set and the ambiguity support would take the form of an infinite ``plus'' shape $([-R,R]\times\R)\cup(\R\times[-R,R])$. Even if uniqueness results can be proven, it appears highly difficult to construct a window with the given ambiguity support.\\

We close with a remark concerning the more general setting of STFT phase retrieval in arbitrary LCA groups. While the uncertainty principle is always present, it is important to note that we also took advantage of the fact that the STFT $V_{\hat{g}}f$ is well-defined since $\hat{g}\in\C^d$ (or $\hat{g}\in L^2(\R)$ in the continuous case). Our approach therefore only has a chance to be successfully carried over to another LCA group if the group in question is its own dual.

\printbibliography

\end{document}